\documentclass[reprint,onecolumn,superscriptaddress,showpacs,amsmath,amssymb,aps,pre,floatfix,shortbibliography]{revtex4-2}
\usepackage{amsthm}
\usepackage{amsmath, mathtools}
\usepackage[utf8]{inputenc}	
\usepackage[english]{babel}
\usepackage{amsmath}
\usepackage{graphicx}		
\usepackage{natbib}
\usepackage{textcomp}
\usepackage{gensymb}
\usepackage[usenames,dvipsnames,svgnames,table]{xcolor}
\usepackage[hidelinks,colorlinks=false,urlcolor=Cerulean,citecolor=black]{hyperref}
\usepackage{siunitx}
\usepackage{mathrsfs}
\usepackage{multirow}
\usepackage{bm}
\usepackage{xfrac}
\usepackage{comment}
\usepackage[normalem]{ulem}

\theoremstyle{plain}
\newtheorem{thm}{Theorem}[section]

\newtheorem{prop}[thm]{Proposition}

\theoremstyle{definition}

\newtheorem{rem}[thm]{Remark}

\makeatletter
\newcommand{\neutralize}[1]{\expandafter\let\csname c@#1\endcsname\count@}
\makeatother

\DeclareMathOperator{\C}{\mathbb{C}}

\DeclareMathOperator{\Spec}{{\rm  Spec}}

    \DeclareFontFamily{U}{wncy}{}
    \DeclareFontShape{U}{wncy}{m}{n}{<->wncyr10}{}
    \DeclareSymbolFont{mcy}{U}{wncy}{m}{n}
    \DeclareMathSymbol{\Sha}{\mathord}{mcy}{"58} 

\numberwithin{equation}{section}

\renewcommand{\i}{\mathrm{i}}

\DeclareSymbolFont{bbold}{U}{bbold}{m}{n}
\DeclareSymbolFontAlphabet{\mathbbold}{bbold}
\newcommand{\ind}{\bm{1}}

\newcommand{\revision}[1]{{\color{black}#1}}

\begin{document}

\title{Broadcasting solutions on \revision{networked systems} of phase oscillators}

\author{Tung T. Nguyen}
\thanks{These authors contributed equally}
\affiliation{Department of Mathematics, Western University, London, ON, Canada}
\affiliation{Western Academy for Advanced Research, Western University, London, ON, Canada}
\affiliation{Western Institute for Neuroscience, Western University, London, ON, Canada}
\author{Roberto C. Budzinski}
\thanks{These authors contributed equally}
\affiliation{Department of Mathematics, Western University, London, ON, Canada}
\affiliation{Western Academy for Advanced Research, Western University, London, ON, Canada}
\affiliation{Western Institute for Neuroscience, Western University, London, ON, Canada}
\author{Federico W. Pasini}
\affiliation{Department of Mathematics, Western University, London, ON, Canada}
\author{Robin Delabays}
\affiliation{University of Applied Sciences and Arts of Western Switzerland HES-SO, Sion, Switzerland}
\affiliation{Center for Control, Dynamical Systems, and Computation, University of California, Santa Barbara, CA, USA}
\author{J\'an Min\'a{\v c}}
\affiliation{Department of Mathematics, Western University, London, ON, Canada}
\affiliation{Western Academy for Advanced Research, Western University, London, ON, Canada}
\author{Lyle E. Muller}
\affiliation{Department of Mathematics, Western University, London, ON, Canada}
\affiliation{Western Academy for Advanced Research, Western University, London, ON, Canada}
\affiliation{Western Institute for Neuroscience, Western University, London, ON, Canada}

\begin{abstract}
\revision{Networked systems} have been used to model and investigate the dynamical behavior of a variety of systems. \revision{For these systems, different levels of complexity can be considered in the modeling procedure.} On one hand, \revision{this can offer} a more realistic and rich modeling option. On the other hand, it can lead to intrinsic difficulty in analyzing the system. Here, we present an approach to investigate the dynamics of Kuramoto oscillators on \revision{networks with different levels of connections: a network of networks.} \revision{To do so, we utilize a construction in network theory known as} the join of \revision {networks}, which represents \revision{``intra-area" and ``inter-area"} connections. This approach provides a reduced representation of the original, \revision{multi-level system}, where both systems have equivalent dynamics. Then, we can find solutions for the reduced system and broadcast them to the \revision{original network of networks}. Moreover, using the same idea we can investigate the stability of these states, where we can obtain information on the Jacobian of the \revision{multi-level system} by analyzing the reduced one. This approach is general for arbitrary connection schemes between nodes within the same \revision{area}. Finally, our work opens the possibility of studying the dynamics of \revision{networked systems} using a simpler representation, thus leading to a better understanding of the dynamical behavior of these systems.
\end{abstract}

\maketitle

\section{Introduction}\label{sec:introduction}

Networked systems have been used to model and understand collective behaviors of several systems \cite{watts1998collective,strogatz2001exploring, boccara2010modeling}. Examples can be found spanning from neuroscience \cite{bassett2018nature} to ecology \cite{banerjee2016chimera} to physical systems, like power-grids \cite{kinney2005modeling, motter2013spontaneous, dorfler2013synchronization}. \revision{These systems} can be understood as single, undirected small networks, or networks of networks, or even networks with higher-order interactions \cite{battiston2020networks,battiston2021physics}. In this context, one important approach to improve the modeling of real systems is given \revision{by networks with different levels of connections \cite{boccaletti2014structure,kivela2014multilayer, kenett2015networks,gao2012networks}}. This approach allows the study of important systems such as \revision{spreading} processes \cite{de2016physics}, animal behavior \cite{silk2018can}, neural dynamics \cite{braun2015dynamic,bassett2018nature}, and synchronization phenomena \cite{zhang2015explosive,della2020symmetries}. On the other hand, the addition of an extra level of connections makes the analysis process more complicated, especially for analytical and mechanistic insights. In this paper, we \revision{present} an approach that offers a simplified way to analyze the dynamics of \revision{these networks}, which allows us to obtain equilibrium points and analyze the stability for a large class of such systems.

\revision{Here, we focus particularly on systems that can be understood as a network of networks \cite{kenett2015networks,gao2012networks,gao2011robustness}. In this case, the system can have two levels of connections: on the first level, there are connections within each sub-network; on the higher level, there are connections between the sub-networks. This is similar to multilayer networks \cite{boccaletti2014structure,kivela2014multilayer}. In our case, though, we consider the same nature of coupling in all levels, i.e. the function describing the interaction is the same. We then consider different scale factors and connection architectures in each level. Due to these reasons, we use the two terms a network of networks and a multi-level network interchangeably throughout this article.}

Based on the framework developed in \cite{doan2022joins,doan2022sprectrum,chebolu2022joins}, we can represent \revision{a network of networks, a multi-level system,} by a join of matrices containing information about the connections \revision{in different levels}. We then introduce a reduced representation of the system, which has a clear connection with the join matrix that represents the \revision{multi-level} network. Thus, this reduced system allows us to study the dynamics of the \revision{sophisticated} system in a simpler way. Then, we can find solutions for the reduced system and broadcast them to the \revision{original} networks. \revision{In other words, we can use our reduced representation of the dynamical system to find solutions and then spread this dynamics to the multi-level network.} \revision{This kind of approach has been used in the literature for the study of the dynamical behavior of these networks under different perspectives. For instance, Schaub \textit{et al.} have introduced a comprehensive framework to study multilayer networks using a similar idea \cite{schaub2016graph}. Further, cluster synchronization and its stability have been studied on these networks considering phase oscillators \cite{menara2019stability, tiberi2017synchronization}.}

\revision{Here, among various applications, our approach offers a direct way to find and analyze equilibrium points in a network of networks with an arbitrary number of sub-networks and arbitrary connection architecture (adjacency matrices) for nodes within the same sub-network. To study this problem,} we consider the individual dynamics of each node given by the Kuramoto model \cite{kuramoto1975self,acebron2005kuramoto,rodrigues2016kuramoto}, so each node is represented by a phase oscillator. The Kuramoto model is a central model for the studying of synchronization and the collective behavior of several systems. It has been used in problems spanning from social interactions \cite{pluchino2006compromise} to biology \cite{breakspear2010generative,bick2020understanding} and physics \cite{arenas2008synchronization, boccaletti2002synchronization}. Particularly, a variety of synchronization phenomena and spatiotemporal patterns has been observed in Kuramoto systems: phase synchronization \cite{rodrigues2016kuramoto,strogatz2000kuramoto}, remote synchronization \cite{nicosia2013remote}, twisted states \cite{townsend2020dense,delabays2016multistability}, chimera states \cite{abrams2004chimera, parastesh2021chimeras}, and Bellerophon states \cite{xu2018origin}.

Furthermore, by considering the properties of joins of matrices and the idea of representing the \revision{multi-level} network by the reduced system, we can investigate into the stability of equilibrium points for Kuramoto oscillators on \revision{these} networks. When the solutions for the \revision{multi-level} systems are obtained through the broadcasting process from the reduced system, we have found a relation between the Jacobian of both systems. In this sense, we can evaluate the spectrum of a simpler matrix - representing the reduced systems - and obtain information on the stability of the equilibrium points for the \revision{more sophisticated, multi-level networks}.

In this paper, we first introduce Kuramoto oscillators on \revision{ a network of networks} and expose our approach with the reduced system and also show how we broadcast solutions to the \revision{multi-level} system - Sec. \ref{sec:main_idea_kuramoto_multilayer}; we extend the approach to the linear stability analysis of \revision{multi-level} networks - in Sec. \ref{sec:stability_analysis}; in addition to this, we show numerical simulations and examples of Kuramoto \revision{network of networks} and their reduced version, which corroborate the main approach introduced in this paper - Sec. \ref{sec:simulations}; finally, we discuss the results presented here and their implications, and draw our conclusions - Sec. \ref{sec:conclusions}.

\revision{\section{Kuramoto oscillators on a network of networks}}\label{sec:main_idea_kuramoto_multilayer}

The Kuramoto model \cite{kuramoto1975self,acebron2005kuramoto,rodrigues2016kuramoto} can be defined as the dynamical system governed by the equation
\begin{equation}
\frac{d\theta_i}{dt} = \omega_{i} + \sum_{j=1}^{\mathcal{N}} A_{ij} \sin( \theta_j - \theta_i ),
\label{eq:KM_main}
\end{equation}
where $\theta_{i} \in [-\pi, \, \pi] $ is the phase of the $i^{\mathrm{th}}$ oscillator, $\omega_{i}$ is its natural frequency, $\mathcal{N}$ is the number of oscillators in the system, and $A_{ij}$ are the elements of the weighted adjacency matrix that represents the connections of this system. That is, $A_{ij} = 0$ if nodes $i$ and $j$ are not connected, and $A_{ij} > 0$ if nodes $i$ and $j$ are connected. Here, we consider the case where all oscillators have the same natural frequency $\omega_{i} = \omega$ for all $i \in [1,\, \mathcal{N}]$.

\vspace{0.5cm}

\subsection{Kuramoto model}

For \revision{a network of networks}, we can consider two levels of interactions: \revision{intra-area (within each sub-network) and inter-are (between sub-networks)} connections. In this case, the Kuramoto model for the $i^{\mathrm{th}}$ oscillator in the $l^{\mathrm{th}}$ sub-network can be written as:
\begin{equation}
    \frac{d(\bm{\theta}_{l})_{i}}{dt} = \omega + \underbrace{ \sum_{j=1}^{N} (\bm{A}_{ll})_{ij} \sin( (\bm{\theta}_{l})_{j} - (\bm{\theta}_{l})_{i} )}_{\mathrm{intra-area}} + \underbrace{\sum_{k=1, k\neq l}^{M}\sum_{j=1}^{N} (\bm{A}_{lk})_{ij} \sin( (\bm{\theta}_{k})_{j} - (\bm{\theta}_{l})_{i} )}_{\mathrm{inter-area}},
\end{equation}
where $i,j \in [1, \, N]$ and $l,k \in [1, \, M]$, $M$ being the number of \revision{sub-networks} and $N$ the number of nodes in each one. Here, the coupling strength between two oscillators can assume different values based on the \revision{sub-network} these oscillators \revision{belong to}. Moreover, it is important to emphasize that the \revision{intra-area coupling considers the coupling between oscillators within a given sub-network, where we consider that each one has $N$ oscillator; the inter-area coupling considers the coupling between oscillators in different sub-networks.}

In this paper, on the other hand, we consider a different representation of a \revision{network of networks system}. Here, we use properties of joins of matrices to describe Kuramoto oscillators on \revision{multi-level} networks as:
\begin{equation}
\frac{d\theta_i}{dt} = \omega + \sum_{j=1}^{NM} A_{ij} \sin( \theta_j - \theta_i ),
\label{eq:KM_multilayer}
\end{equation}
where now $i \in [1,\, NM]$, such that:
\begin{equation}
    \bm{\theta} = (\underbrace{\theta_{1}, \theta_{2}, \cdots, \theta_{N}}_{1^{\mathrm{st}} \mathrm{area}}, \underbrace{\theta_{N+1}, \theta_{N+2}, \cdots, \theta_{2N}}_{2^{\mathrm{nd}} \mathrm{area}}, \cdots, \underbrace{\theta_{N(M-1) + 1}, \theta_{N(M-1) +2}, \cdots, \theta_{NM}}_{M^{\mathrm{th}} \mathrm{area}}).
    \label{eq:theta_multilayer}
\end{equation}
The adjacency matrix describing the \revision{multi-level} system is the $NM \times NM$ matrix:
\begin{equation}
\bm{A}=\left(\begin{array}{c|c|c|c}
\bm{A}_{11} & \bm{A}_{12} & \cdots & \bm{A}_{1M} \\
\hline
\bm{A}_{21} & \bm{A}_{22} & \cdots & \bm{A}_{2M} \\
\hline
\vdots & \vdots & \ddots & \vdots \\
\hline
\bm{A}_{M1} & \bm{A}_{M2} & \cdots & \bm{A}_{MM}
\end{array}\right),
\label{eq:matrix_multilayer}
\end{equation}
where $\bm{A}_{lk}$ is a $N \times N$ matrix, which describes the connections between the oscillators in the $l^{\mathrm{th}}$ and $k^{\mathrm{th}}$ \revision{areas (sub-networks)}. In case $l = k$, the block
represents the \revision{intra-area connections in the $l^{\mathrm{th}}$ area}. In case $l \neq k$, the block  represents the \revision{inter-area connections between oscillators in areas $l$ and $k$}. Note that in our treatment these matrices absorb the information about the coupling strength. \revision{Here, our approach is able to handle multi-level networks where the coupling strength between two different areas (sub-networks) can be different given two different areas. This means that the coupling strength between the areas $l$ and $k$ can be different than the coupling strength between areas $l$ and $k+1$.} 

For details on the joins of matrices and on the mathematical background of our approach, see \cite{doan2022joins,doan2022sprectrum,chebolu2022joins}.

\subsection{Reduced Kuramoto model and \revision{broadcasted} solutions}

In this paper we focus on networks satisfying a regularity condition: namely that the off-diagonal blocks $\bm{A}_{lk}$ for $l \neq k$ have the same row sum (that is, they are row-regular, or semimagic, matrices, \cite{doan2022sprectrum}). In terms of the Kuramoto model, this assumption means that, for any pair of \revision{areas}, each oscillator in the first \revision{sub-network} is connected with the same number of oscillators in the other \revision{sub-network}. To investigate the dynamics of such a system, we introduce an auxiliary \textit{reduced} $M \times M$ network, in which each \revision{sub-network} of the original system is condensed into one global oscillator. More concretely, the reduced system is described by the adjacency matrix
\begin{equation}
    \bar{\bm{A}}= 
\begin{pmatrix}
r_{\bm{A}_{11}} & r_{\bm{A}_{12}} & \cdots & r_{\bm{A}_{1M}} \\
r_{\bm{A}_{21}} & r_{\bm{A}_{22}} & \cdots & r_{\bm{A}_{2M}} \\
\vdots  & \vdots  & \ddots & \vdots  \\
r_{\bm{A}_{M1}} & r_{\bm{A}_{M2}} & \cdots & r_{\bm{A}_{MM}} 
\end{pmatrix},
\label{eq:matrix_reduced_system}
\end{equation}
where $r_{\bm{A}_{lk}}$ is the row sum of $\bm{A}_{lk}$. At this point, it is important to emphasize that these matrices have information about the coupling strength for the \revision{intra-area and inter-area cases}, thus the line sum captures this information.

We can now define the reduced Kuramoto model:
\begin{equation}
\frac{d\bar{\theta}_i}{dt} = \omega + \sum_{j=1}^{M} \bar{A}_{ij} \sin( \bar{\theta}_j - \bar{\theta}_i ),
\label{eq:KM_reduced_system}
\end{equation}
where now $i \in [1,\,M]$, \revision{and $\bar{\theta}_i\in[-\pi,\pi]$ is the phase of an oscillator representing the $i$th sub-network}. We emphasize that the main diagonal of $\bar{\bm{A}}$ does not affect the dynamics \revision{of the system}, since $\sin(\bar{\theta}_{i} - \bar{\theta}_{i}) = 0$ \revision{thus not having an effect on the dynamics of the Kuramoto networks}. Therefore, in the following, we often make the harmless identification
\begin{equation}
    \bar{\bm{A}}= 
\begin{pmatrix}
r_{\bm{A}_{11}} & r_{\bm{A}_{12}} & \cdots & r_{\bm{A}_{1M}} \\
r_{\bm{A}_{21}} & r_{\bm{A}_{22}} & \cdots & r_{\bm{A}_{2M}} \\
\vdots  & \vdots  & \ddots & \vdots  \\
r_{\bm{A}_{M1}} & r_{\bm{A}_{M2}} & \cdots & r_{\bm{A}_{MM}} 
\end{pmatrix} \simeq \begin{pmatrix}
0 & r_{\bm{A}_{12}} & \cdots & r_{\bm{A}_{1M}} \\
r_{\bm{A}_{21}} & 0 & \cdots & r_{\bm{A}_{2M}} \\
\vdots  & \vdots  & \ddots & \vdots  \\
r_{\bm{A}_{M1}} & r_{\bm{A}_{M2}} & \cdots & 0
\end{pmatrix}.
\end{equation}

The Kuramoto network on the \revision{multi-level} system - Eq. (\ref{eq:KM_multilayer}) - represented by $\bm{A}$ - Eq. (\ref{eq:matrix_multilayer}) - and the Kuramoto network on the reduced system - Eq. (\ref{eq:KM_reduced_system}) - represented by $\bar{\bm{A}}$ - Eq. (\ref{eq:matrix_reduced_system}) - have the \revision{an equivalent} dynamics. Thus, we can use the reduced system to better understand the dynamics on the more \revision{sophisticated, multi-level} network.

More precisely, solutions for the reduced system can be \textit{broadcasted} to the original network: given a solution of the reduced Kuramoto model:
\begin{equation}
    \bar{\bm{\theta}}^{\ast} = (\bar{\theta}_{1}^{\ast}, \bar{\theta}_{2}^{\ast}, \cdots, \bar{\theta}_{M}^{\ast}),
    \label{eq:solution_reduced_system}
\end{equation}
then
\begin{equation}
\bm{\theta}^{\ast} = (\underbrace{\bar{\theta}_{1}^{\ast}, \bar{\theta}_{1}^{\ast}, \cdots, \bar{\theta}_{1}^{\ast}}_{1^{\mathrm{st}} \mathrm{area}}, \underbrace{\bar{\theta}_{2}^{\ast}, \bar{\theta}_{2}^{\ast}, \cdots, \bar{\theta}_{2}^{\ast}}_{2^{\mathrm{nd}} \mathrm{area}}, \cdots, \underbrace{\bar{\theta}_{M}^{\ast}, \bar{\theta}_{M}^{\ast}, \cdots, \bar{\theta}_{M}^{\ast}}_{M^{\mathrm{th}} \mathrm{area}}),
\label{eq:solution_multilayer_spread}
\end{equation}
is the corresponding \textit{broadcasted} solution of the \revision{multi-level} system described by $\bm{A}$.

Moreover, we note the dependence of the dynamics on the coupling strength. The \revision{inter-area} coupling controlled by $\epsilon$ is also presented in the reduced system, see Eq. (\ref{eq:matrix_reduced_system}). So, if we consider \revision{an equivalent} initial state, both systems evolve to the \revision{equivalent} final state. However, the stronger the coupling, the faster the transition, which we summarize in the following proposition whose proof follows from direct calculations (valid for nonzero coupling only):
\begin{prop}\label{prop:coupling_scale}
Suppose $\bar{\bm{\theta}}^{\ast}=(\bar{\theta}_{1}^{\ast}, \bar{\theta}_{2}^{\ast}, \cdots, \bar{\theta}_{M}^{\ast})$ is a solution of the Kuramoto model
\[ \dot{\theta}_i = \epsilon \sum_{j} A_{ij} \sin(\theta_j-\theta_i) ,\]
with intial condition $\bm{\theta}(0)=\bm{\theta}_0.$
Let $\bm{\psi}(t)=\bar{\bm{\theta}}^{\ast}\big(\frac{\epsilon'}{\epsilon} t\big)$. Then $\bm{\psi}(t)$ is a solution of the Kuramoto model
\[ \dot{\psi}_i = \epsilon' \sum_{j} A_{ij} \sin(\psi_j-\psi_i) ,\]
with the same initial condition $\bm{\psi}(0)=\bm{\theta}_0.$
\end{prop}

\begin{rem}\label{rem:different_nodes}
We note that the idea presented in this section about the correspondence between the \revision{network of networks} and the reduced system is valid even when the \revision{sub-networks} have a different number of nodes. In this case, $\bm{A}$ would be a $\mathcal{M} \times \mathcal{M}$ matrix, where $\mathcal{M} = N_{1} + N_{2} + \cdots + N_{M}$. We presented here the version where the \revision{sub-network} have the same number of nodes $N$ for simplicity purpose. 
\end{rem}

\subsection{Kuramoto order parameter}\label{sec:order_parameter}

In order to characterize the dynamical behavior of the system, we use the Kuramoto order parameter \cite{rodrigues2016kuramoto,acebron2005kuramoto}. The order parameter for the \revision{multi-level} system is defined as
\begin{equation}
    R(t) = \frac{1}{NM} \left|\sum\limits_{j=1}^{NM} \exp{(\i\theta_j(t))}\right|,
    \label{eq:order_parameter_multilayer}
\end{equation}
where $\theta(t)$ is given by Eq. (\ref{eq:KM_multilayer}). Here, $R(t) = 1$ means that all oscillators in all \revision{sub-networks} have the same phase at a given time $t$, which is defined as phase synchronization. If the oscillators are not phase synchronized, the order parameter is less than $1$. For the asynchronous behavior, $R$ assumes residual values. When the system is on a phase-locking solution, also called a twisted state, $R(t) = 0$.

We also measure the level of synchronization of the Kuramoto network on the reduced representation using the Kuramoto order parameter. In this case, the order parameter is written as
\begin{equation}
    \bar{R}(t) = \frac{1}{M} \left|\sum\limits_{j=1}^{M} \exp{(\i \bar{\theta}_j(t))}\right|,
    \label{eq:order_parameter_reduced}
\end{equation}
where $\bar{R}(t) = 1$ means the reduced system is phase synchronized at time $t$, an asynchronous behavior leads to a residual value of $\bar{R}$, and if the reduced system has a phase-locking solution (``twisted state"), then $\bar{R} = 0$. 

By the definition of the broadcasted solution and direct calculations, we have the following proposition.
\begin{prop}\label{prop:order_parameter}
Suppose that $\bar{\bm{\theta}}^{\ast}$ is a solution of the Kuramoto model on $\bar{\bm{A}}$ and $\bm{\theta}^{\ast}$ is the corresponding broadcasted solution of the Kuramoto model on $\bm{A}.$ Then the two systems have the same Kuramoto order parameter for all times: $R(t) = \bar{R}(t), \, t \geq 0.$
\end{prop}

\section{Stability of equilibrium points}\label{sec:stability_analysis}

We now propose to analyze the existence and stability of some equilibria for \revision{a network of networks} through our knowledge of the existence and stability properties of equilibria in the reduced system. Of course, there is a loss in complexity when considering the reduced system and we cannot guarantee that its analysis provides a full picture of the whole \revision{multi-level} system. However, we show that some crucial information can be obtained from the reduced system by simply broadcasting a known solution thereof to the \revision{multi-level} network. 

In this section, we require that connections between each pair of connected \revision{areas} are dense, which we approximate by a mean-field (or all-to-all) coupling. In other words, we assume that if layers $l$ and $k$ are connected, then each node in \revision{area} $l$ is connected to each node in \revision{area} $k$. 
\revision{As we demonstrate later, we need to consider this assumption to utilize some known results from spectral graph theory when we analyze the correspondence between the multi-level system and the reduced one.} Whereas these assumptions formally restrict the applicability of our method, they still cover a rather wide class of networked systems. It is furthermore quite standard to consider identical coupling strength within a networked system and to approximate unknown or uncertain couplings by mean-field interactions~\cite{bovier2006statistical,chayes2009mean,kadanoff2009more,strogatz1988phase-locking,dorfler2014synchronization}. Formally, the above assumptions translate as:
\begin{enumerate}
    \item $\bm{A}_{ll}$ is the adjacency matrix of an undirected, connected, weighted graph with positive weights. 
    \item For $l \neq k$, $\bm{A}_{lk}= \epsilon_{lk} \bm{1}_{N}$ where $\bm{1}_{N}$ is a matrix $N \times N$ with all entries equal to $1$. 
\end{enumerate}

As observed in the previous section, the (simpler) reduced system can be used to study the dynamics of the \revision{multi-level} network. In particular, equilibrium points of the reduced system can be broadcast to equilibrium points of the \revision{multi-level} one.

However, for stability analysis, we cannot extend the idea in a straightforward way. The reason is that a neighborhood of an equilibrium point for the \revision{network of networks} $\bm{\theta}^{\ast}$ is $NM$-dimensional, while a neighborhood of an equilibrium point for the reduced system $\bm{\bar{\theta}}^{\ast}$ is $M$-dimensional. Broadcasting the neighborhood of $\bar{\bm \theta}^{\ast}$
\begin{equation}
    B_{\varepsilon}(\bar{\bm \theta}^{\ast}) = \left\{\bar{\bm \theta}^{\ast} + \bm{\varepsilon} ~|~ \|\bm{\varepsilon}\|_2<\varepsilon\right\}\, ,
\end{equation}
to the \revision{multi-level} system yields an $M$-dimensional subset of the $NM$-dimensional neighborhood $B_\varepsilon(\bm{\theta})$. 

In order to perform a stability analysis of the broadcasted fixed point, one needs a more in-depth analysis of the Jacobian matrix. Fortunately, we show that such an analysis can still be performed, and leverages our knowledge of the fixed point for the reduced system. 

In order to study the stability of these solutions we can use the Jacobian of each system. Here we denote the Jacobian for the \revision{multi-level} system at the equilibrium point $\bm{\theta}^{\ast}$ as $J(\bm{\theta}^{\ast}) = J_{\bm{A}}(\bm{\theta}^{\ast})$, and the Jacobian for the reduced system at the equilibrium point $\bar{\bm{\theta}}^{\ast}$ as $J(\bar{\bm{\theta}}^{\ast}) = J_{\bar{\bm{A}}}(\bar{\bm{\theta}}^{\ast})$.

We first recall the matrix $\bar{\bm{A}}$, defined in Eq. (\ref{eq:matrix_reduced_system}). Then, we notice that main diagonal of $\bar{\bm{A}}$ does not affect the dynamics of the system (since $\sin(\bar{\theta}_{i} - \bar{\theta}_{i}) = 0$), so we can write the matrix as
\begin{equation}
    \bar{\bm{A}}= 
\begin{pmatrix}
0 & r_{\bm{A}_{12}} & \cdots & r_{\bm{A}_{1M}} \\
r_{\bm{A}_{21}} & 0 & \cdots & r_{\bm{A}_{2M}} \\
\vdots  & \vdots  & \ddots & \vdots  \\
r_{\bm{A}_{M1}} & r_{\bm{A}_{M2}} & \cdots & 0 
\end{pmatrix},
\end{equation}
where, by definition, $r_{\bm{A}_{ij}}=N \epsilon_{ij}$. Based on this matrix, we can now write the Jacobian for the reduced system as
\begin{equation}
J(\bar{\bm{\theta}}^{\ast})=\begin{pmatrix}
-\lambda_1 & r_{\bm{A}_{12}}\cos(\bar{\theta}^{\ast}_{2}-\bar{\theta}^{\ast}_{1}) & \cdots & r_{\bm{A}_{1M}}\cos(\bar{\theta}^{\ast}_{M}-\bar{\theta}^{\ast}_{1}) \\
r_{\bm{A}_{21}}\cos(\bar{\theta}^{\ast}_{1}-\bar{\theta}^{\ast}_{2}) & -\lambda_2 & \cdots & r_{\bm{A}_{2M}}\cos(\bar{\theta}^{\ast}_{M}-\bar{\theta}^{\ast}_{2}) \\
\vdots  & \vdots  & \ddots & \vdots  \\
r_{\bm{A}_{M1}}\cos(\bar{\theta}^{\ast}_{1}-\bar{\theta}^{\ast}_{M}) & r_{\bm{A}_{M2}}\cos(\bar{\theta}^{\ast}_{2}-\bar{\theta}^{\ast}_{M}) & \cdots &  -\lambda_{M} 
\end{pmatrix},
\label{eq:jacobian_reduced system}
\end{equation}
where
\begin{equation}
    \lambda_i =\sum_{j=1, j\neq i}^M r_{\bm{A}_{ij}}\cos(\bar{\theta}^{\ast}_{j}-\bar{\theta}^{\ast}_{i}).
    \label{eq:lambda}
\end{equation}
In particular, the above Jacobian matrices are semimagic square matrices with line sum zero. 

In order to compute the Jacobian for the \revision{multi-level} system $J(\bm{\theta}^{\ast})$, we first recall the definition of the Laplacian. For an undirected weighted graph $H$ with adjacency matrix $\bm{B}_{H}$ and degree matrix $\bm{D}_H$, the Laplacian is defined as:
\begin{equation}
    \bm{L}_{H} = \bm{D}_{H} - \bm{B}_{H}.
    \label{eq:laplacian_definition}
\end{equation}
We note that, as long as $H$ is undirected, $\bm{L}_{H}$ is always a symmetric semi-magic square matrix (with row sum equal to zero). Regarding the spectrum of $\bm{L}_H$, we have the following proposition. \revision{We emphasize that this proposition has been studied and demonstrated already in \cite{mohar1991laplacian}. We show this proposition and its proof here for completeness.}
\begin{prop} \label{prop:laplacian}
Let $\bm{L}_{H}$ be the weighted Laplacian matrix of an undirected, connected, weighted graph $H$. Then $\bm{L}_{H}$ has a unique zero eigenvalue and all other eigenvalues are strictly positive. 
\end{prop}
\begin{proof}
By Gershgorin's Discs Theorem \cite[Theroem 6.1.1]{horn2012matrix}, all eigenvalues of $\bm{L}_H$ are nonnegative. Note also that the eigenvectors of $\bm{L}_H$ form an orthonormal basis of $\mathbb{R}^n$, because $\bm{L}_{H}$ is symmetric. Now assume that $\bm{v}$ is an eigenvector of $\bm{L}_H$ associated with the eigenvalue $0$. 
Then,
\begin{align}
    0 &= \bm{v}^\top\bm{L}_H\bm{v} = \sum_{i<j}a_{ij}(v_i-v_j)^2\, ,
\end{align}
where $a_{ij}$ is the weight of the edge linking nodes $i$ and $j$. Each term in the sum above is necessarily zero. Therefore, for any pair of nodes $(i,j)$ that are connected, $v_i=v_j$. As the graph is connected, the eigenvector $\bm{v}$ is necessarily a multiple of the vector $(1,\ldots, 1)^\top$. Hence $\bm{L}_H$ has a unique zero eigenvalue and all other eigenvalues are strictly positive. 
\end{proof}

We now compute the Jacobian $J(\bm{\theta}^{\ast})$. Here, $G_{i}$ is the graph representing the $i^{\mathrm{th}}$ \revision{area}, and $\epsilon_{ij}$ is the coupling strength between nodes in the \revision{area} $i$ and $j$. By definition, for two indices $(i_1,j_1)$ such that $i_1 \in V(G_i), j_1 \in V(G_j)$ where $i \neq j$ then
\begin{equation}
[J(\bm{\theta}^{\ast})]_{i_1, j_1} = A_{i_1, j_1} \cos(\bar{\theta}_{j}^{\ast}-\bar{\theta}_{i}^{\ast})= \epsilon_{ij} \cos(\bar{\theta}_{j}^{\ast}-\bar{\theta}_{i}^{\ast}).
\end{equation}
We emphasize that for the stability analysis, the coupling between \revision{areas} is considered uniform, i.e., oscillators in the \revision{area} $i$ are connected to oscillators in the \revision{area} $j$ with the same coupling strength.

If $i_1, i_2 \in V(G_i)$ and $i_1 \neq i_2,$ then 
\begin{equation}
[J(\bm{\theta}^{\ast})]_{i_1, i_2} = (\bm{A}_{ii})_{i_1, i_2} \cos(\bar{\theta}_{i}^{\ast}-\bar{\theta}_{i}^{\ast}) = (\bm{A}_{ii})_{i_1, i_2}. 
\end{equation}
Finally, we need to consider the case $i_1=i_2$ and $i_1 \in V(G_i)$. For this part, we observe that $J(\bm{\theta^}{\ast})$ is a semimagic square matrix with line sums equal to zero. We can then see that
\begin{equation}
[J(\bm{\theta}^{\ast})]_{i_1, i_1}= -\lambda_i- \deg(i_1).
\end{equation}
with $\deg(i_1)$ denoting the weighed degree of node $i_1$ \revision{and $\lambda_{i}$ being defined in Eq. (\ref{eq:lambda}).} 

By combining these facts, we can write the Jacobian for the \revision{network of networks} at the equilibrium point $\bm{\theta}^{\ast}$ as:
\begin{equation}
J(\bm{\theta}^{\ast}) = J_{\bm{A}}(\bm{\theta}^{\ast}) = \left(\begin{array}{c|c|c|c}
-\lambda_1 \bm{I} -\bm{L}_{\bm{A}_{11}} & \epsilon_{12} \cos(\bar{\theta}_{2}^{\ast}-\bar{\theta}_{1}^{\ast})) \ind & \cdots & \epsilon_{1M} \cos(\bar{\theta}_{M}^{\ast}-\bar{\theta}_{1}^{\ast})\ind \\
\hline
\epsilon_{21} \cos(\bar{\theta}_{1}^{\ast}-\bar{\theta}_{2}^{\ast}) \ind & -\lambda_2 \bm{I} - \bm{L}_{\bm{A}_{22}} & \cdots &\epsilon_{2M} \cos(\bar{\theta}_{M}^{\ast}-\bar{\theta}_{2}^{\ast}) \ind \\
\hline
\vdots & \vdots & \ddots & \vdots \\
\hline
\epsilon_{M1} \cos(\bar{\theta}_{1}^{\ast}-\bar{\theta}_{M}^{\ast}) \ind & \epsilon_{M2} \cos(\bar{\theta}_{2}^{\ast}-\bar{\theta}_{M}^{\ast}) \ind & \cdots & -\lambda_M \bm{I} - \bm{L}_{\bm{A}_{MM}}
\end{array}\right).
\label{eq:jacobian_multilayer_system}
\end{equation}
where, we recall that $\epsilon_{ij}$ is the coupling between oscillators in \revision{areas} $i$ and $j$.

One realizes, in particular, that $J(\bm{\theta}^{\ast})$ is a joined union of semimagic square matrices (see \cite{doan2022joins, doan2022sprectrum} for details). Furthermore, we have the following equality 
\begin{equation}\label{eq:Jbar}
    \overline{J_{\bm{A}}(\bm{\theta}^{\ast})}= J_{\bar{\bm{A}}}(\bar{\bm{\theta}}^{\ast}).
\end{equation}
It is important to emphasize that Eq.~\eqref{eq:Jbar} is valid if the equilibrium point $\bm{\theta}^{\ast}$ of the \revision{network of networks} is obtained through the broadcasting of the solution $\bar{\bm{\theta}}^{\ast}$ of the reduced system - see Eqs. (\ref{eq:solution_reduced_system}) and (\ref{eq:solution_multilayer_spread}).

We summarize these ideas in the following proposition:
\begin{prop} \label{prop:jacobian}
Suppose $\bar{\bm{\theta}}^{\ast}$ is an equilibrium point of the Kuramoto model associated with the reduced system ($\bar{\bm{A}}$). Then $\bm{\theta}^{\ast}$ is an equilibrium point of the Kuramoto model on the \revision{multi-level} system ($\bm{A}$) following Eq. (\ref{eq:solution_multilayer_spread}). In addition to that, suppose that the matrices representing the inter-\revision{areas} connections ($\bm{A}_{ij}, \, i \neq j$) have all entries the same. The Jacobian $J_{\bar{\bm{A}}}(\bar{\bm{\theta}}^{\ast})$ and $J_{\bm{A}}(\bm{\theta}^{\ast})$ are given by Eq. (\ref{eq:jacobian_reduced system}) and Eq. (\ref{eq:jacobian_multilayer_system}), respectively.
Furthermore, we have the following relation 
\[ \overline{J_{\bm{A}}(\bm{\theta}^{\ast})}= J_{\bar{\bm{A}}}(\bar{\bm{\theta}}^{\ast}) .\] 
\end{prop}

For the next proposition, we need to introduce a couple of notations. 
For a set $S=\{s_1, s_2, \ldots, s_m \} \subset \C^m$ and $a, b \in \C$, we denote 
\[ aS - b= \{as_1-b, as_2-b, \ldots, as_m-b \} .\] 
We also denote by $S \setminus s$ the set difference of $S$ by $s.$ 
By \cite[Theorem 3.3]{doan2022sprectrum}, we have the following statement about the spectrum of $J_{\bm{A}}(\bm{\theta}^{\ast})$:
\begin{prop} \label{prop:spectrum}
The spectrum of $J_{\bm A}(\bm{\theta}^{\ast})$ can be defined in terms of:
\begin{itemize}
    \item The spectra of the Laplacian matrices of each \revision{area}, $\bm{L}_{\bm{A}_i}$;
    \item The Jacobian of the reduced system, $J_{\bar{\bm A}}(\bar{\bm \theta}^{\ast})$. 
\end{itemize}
Namely, 
\begin{equation}
    \Spec(\bm{J}_{\bm{A}}(\bm{\theta}^{\ast})) = \bigcup_{i=1}^d \left\{(-\Spec(\bm{L}_{\bm{A}_i}) \setminus \{0\}) - \lambda_i\right\} \cup \Spec(\bm{J}_{\bar{\bm{A}}}(\bar{\bm \theta}^{\ast}))\, .
\end{equation}
\end{prop} 

In Sec. \ref{sec:main_idea_kuramoto_multilayer}, we showed how we can use the reduced system to broadcast solutions to the \revision{Kuramoto model on a network of networks}. We can now study the stability of these solutions.
\begin{prop}
The fixed point $\bar{\bm \theta}^{\ast}$ for the reduced system is linearly stable if and only if the corresponding broadcasted fixed point for the \revision{multi-level} system is linearly stable. 
\end{prop}
\begin{proof}
Assume that $\bar{\bm{\theta}}^{\ast}$ is linearly stable, then $J_{\bar{A}}(\bar{\bm{\theta}}^{\ast})$ is symmetric negative-semidefinite. By Sylvester's criterion \cite[Theorem 3.3.12]{bazaraa2013nonlinear}, we know that $\lambda_i \geq 0.$ Furthermore, by Proposition \ref{prop:laplacian}, we know that all eigenvalues of $\bm{L}_{\bm{A}_i}$, except 0, are positive. These two facts show that all elements in 
\[ \bigcup_{i=1}^d \left\{ (-\Spec(\bm{L}_{\bm{A}_i}) \setminus \{0\})-\lambda_i  \right \}, \]
are negative. By Proposition \ref{prop:spectrum} and the assumption that $\bar{\bm{\theta}}^{\ast}$ is linearly stable, we conclude that all eigenvalues of $J_{\bm A}(\bm{\theta}^{\ast})$, other than $0$, must be negative. Hence $\bm{\theta}^{\ast}$ is linearly stable. 

The other direction is straightforward. Indeed, by Proposition~\ref{prop:spectrum}, the spectrum of $\bm{J}_{\bar{A}}(\bar{\bm \theta}^{\ast})$ is a subset of the spectrum of $\bm{J}_{\bm{A}}(\bm{\theta}^{\ast})$. Therefore, if $\bm{\theta}^{\ast}$ is linearly stable, then $\bar{\bm \theta}^{\ast}$ is linearly stable as well. 
\end{proof}

\section{Numerical simulations and examples}\label{sec:simulations}

In this section, we show examples and numerical simulations of \revision{Kuramoto oscillators on a network of networks} and its respective reduced version that we introduce in this paper. Figure \ref{fig:schematic_representation} shows a graphic representation of this approach. We study Kuramoto oscillators on \revision{multi-level} networks, where each \revision{area} is composed of $N$ nodes and has an internal connection structure and coupling (intra-coupling). These \revision{areas} ($M$ in total) are then coupled with an external connection structure (inter-coupling), thus leading to a two-levels coupling scheme (left). In this paper, we introduced a reduced ``inter-\revision{area}" representation that allows us to investigate the dynamics of the \revision{multi-level} system in a simplified way. In this simpler representation, each \revision{area} is given by a node, which is connected to other nodes following the inter-coupling scheme (right). As shown in Sec. \ref{sec:main_idea_kuramoto_multilayer}, the reduced system and the \revision{network of networks} have equivalent dynamics and we can broadcast solutions from the former to the latter.
\begin{figure}[htb]
    \centering
    \includegraphics[width=0.65\columnwidth]{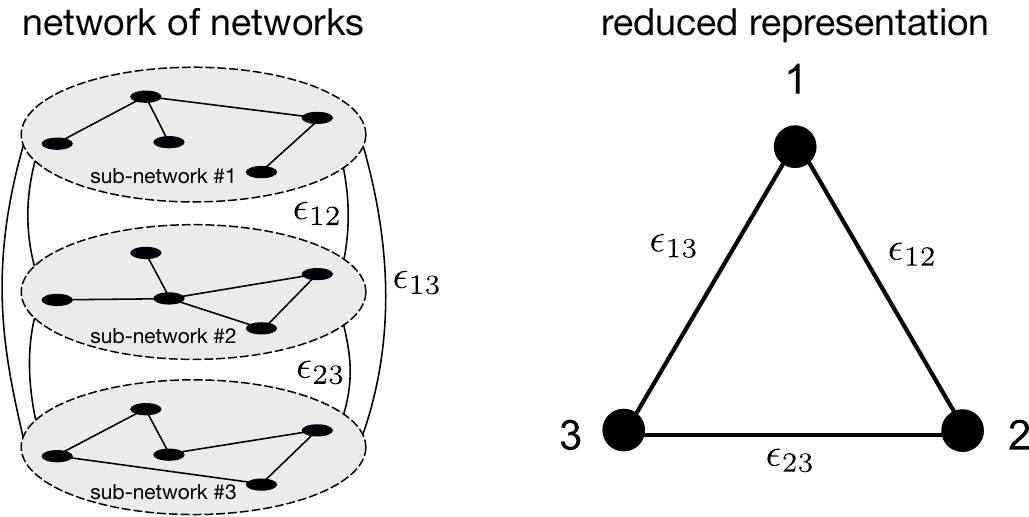}
    \caption{\textbf{Graphic representation of the system.} The \revision{network of networks} is given by $M$ \revision{sub-networks} with $N$ nodes each one. Here, we use a two-level system, where each \revision{sub-network (area)} has internal connections in addition to the inter-\revision{area} connections. Our approach allows us to study the dynamics of these systems using a reduced $M \times M$ system. This figure depicts an example with $M = 3$ \revision{sub-networks}.}
    \label{fig:schematic_representation}
\end{figure}

 \subsection{Phase synchronization}

We first consider a \revision{network of networks} with $M = 3$, where each \revision{sub-network} is composed of $N = 100$ nodes. The internal connection structure of each \revision{sub-network} is given by \revision{random networks. Particularly, we consider Erd\H{o}s-R\'enyi matrices with probability given by $0.15$}. The adjacency matrix for the internal coupling is given by $A_{ij, \, \text{intra}} = 1$ if nodes $i$ and $j$ are connected, and by $A_{ij, \, \text{intra}} = 0$ otherwise\revision{, while the internal coupling strength is given by $\epsilon_{\mathrm{intra}} = 1.0$}. The inter-layers coupling scheme is given by a complete graph, where all nodes in one layer are connected to all nodes in another layer with coupling strength \revision{that varies randomly between each pair of sub-networks in the interval of $\epsilon_{\mathrm{inter}} \in [0.005, 0.01]$}.

\begin{figure*}[htb]
    \centering
    \includegraphics[width=0.9\textwidth]{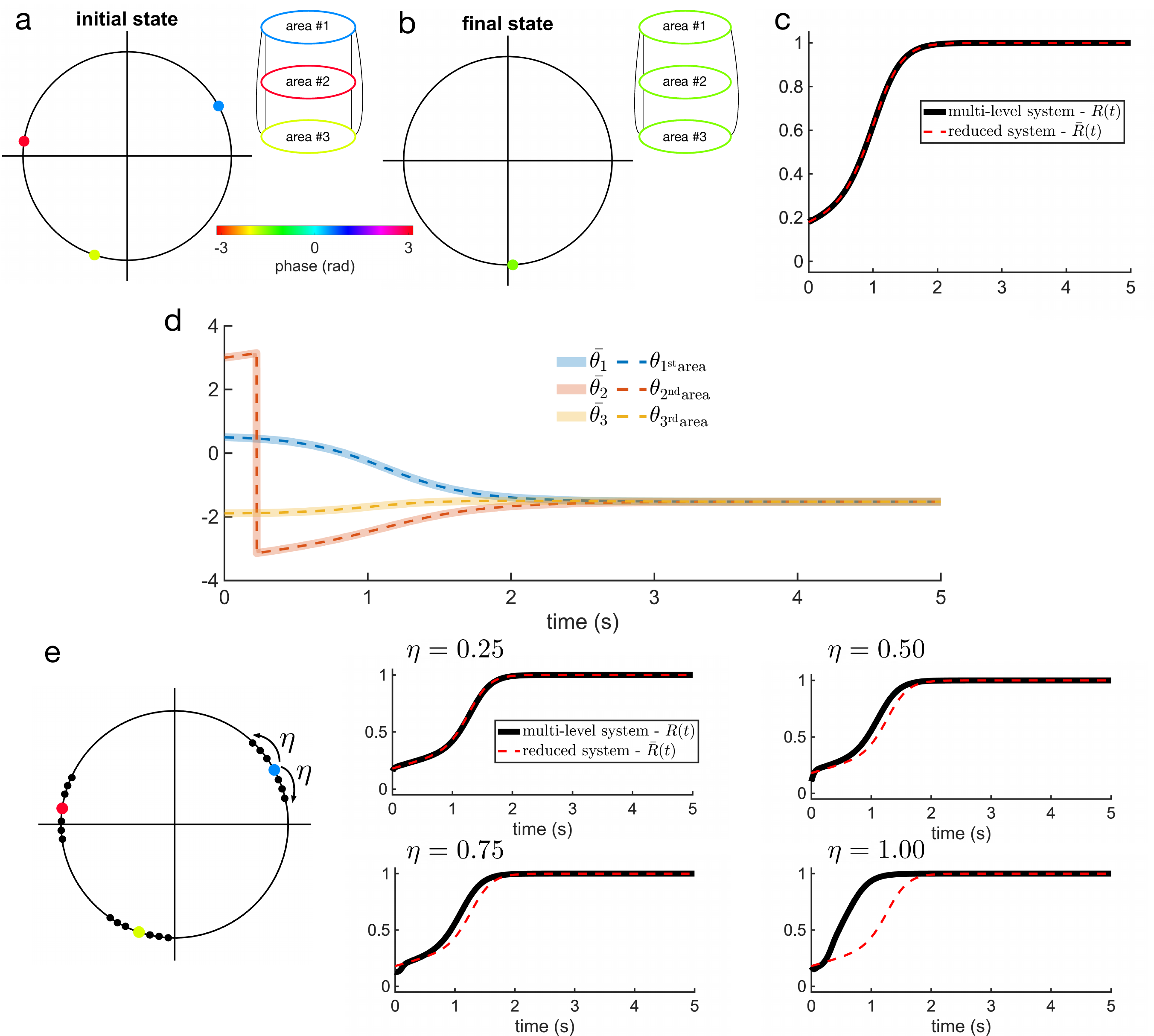}
    \caption{\textbf{Dynamics on a \revision{network of networks} and its reduced system with random initial conditions.} In the beginning of the simulation, each \revision{area} is phase synchronized but on different phases \textbf{(a)}, therefore the \revision{multi-level} system is nonsynchronized. \textbf{(b)} As time evolves, due to the inter-\revision{area} coupling, the internal dynamics changes and the \revision{multi-level} system transition to phase synchronization. \textbf{(c)} Here, the dynamics of the \revision{multi-level} system and the reduced one are the same, which is expressed by the Kuramoto order parameter $R(t)$. \textbf{(d)} The trajectories for both the reduced system and the \revision{multi-level show that}: (i) these systems have equivalent dynamics; and (ii) the network transitions to phase synchrony as time evolves. \revision{\textbf{(e)} Our approach can be used even when the oscillators within a sub-network do not have the same phase. We use the broadcast procedure and then apply a perturbation over the phases of the oscillators with amplitude $\eta$. The reduced system and the network of networks have equivalent dynamics even in the presence of this perturbation; however, if the perturbation is too big, our approach no longer can be applied.}}
    \label{fig:3_layers_random_ic}
\end{figure*}
The first example is depicted in Fig. \ref{fig:3_layers_random_ic}. Here, the initial state is given by randomly selected phases for each \revision{area}, which are represented in the unitary circle in color-code (Fig. \ref{fig:3_layers_random_ic}a). All oscillators in each \revision{area} have the same phase, therefore each \revision{sub-network} is phase synchronized, but the entire, \revision{multi-level} system is not -- in the beginning of the simulation. As time evolves, the \revision{multi-level} system transitions to phase synchrony due to the inter-\revision{area} coupling, which leads all oscillators in all \revision{areas} to depict the same phase in the final state (Fig. \ref{fig:3_layers_random_ic}b).

The change on the dynamics over time is represented by the Kuramoto order parameter (Fig. \ref{fig:3_layers_random_ic}c), which is evaluated through Eq. (\ref{eq:order_parameter_multilayer}) for the \revision{network of networks} and through Eq. (\ref{eq:order_parameter_reduced}) for the reduced system. Here, one can observe that both systems have equivalent dynamics: at $t = 0$, $R$ has a low value due to the random phases on the initial state; due to the coupling, $R(t)$ increases as time evolves and, at $t \approx 2$, the order parameter reaches the unity once both systems reach phase synchronization.

In order to show details on the dynamical equivalence between the reduced system and the \revision{multi-level} network, we plot the trajectories given by the phases of the oscillators in Fig. \ref{fig:3_layers_random_ic}d. Here, the phases in the reduced system are represented by the solid-shaded lines, and the phases of the oscillators in the \revision{multi-level} network are given by the dashed lines. This result emphasizes that: (i) these systems have equivalent dynamics, and also (ii) the network transitions to phase synchrony as time evolves since the oscillators converge to the same phase.

\revision{Furthermore, we can consider the case where the oscillators within each sub-network (area) do not have the same phase at the beginning of the analysis. To do so, we use the same broadcast procedure in addition to a random perturbation with amplitude $\eta$ (Fig. \ref{fig:3_layers_random_ic}e). In this case, the phase of each oscillator within a sub-network is randomly chosen around the phase of this area in the reduced system. Figure \ref{fig:3_layers_random_ic}e shows that the reduced system is able to capture the dynamics of the network of networks even in the presence of perturbation over the initial state. However, if the perturbation is too big, the reduced system and the multi-level network no longer have equivalent dynamics.}

 \subsection{Unstable twisted states}

We use the same system to analyze a different kind of dynamical behavior: phase-locking states, where a constant phase difference is observed across oscillators \cite{townsend2020dense}. This kind of state is also known as ``twisted states", and, for a network with $M$ units, the $p^{\mathrm{th}}$ twisted state is given by:
\begin{equation}
\bm{\theta}^{(p)}=\left(0,  \frac{-2 \pi p}{M}, \cdots, \frac{-2 \pi p (M-1)}{M} \right).  
\label{eq:twisted_states}
\end{equation}
We use this equation to obtain the twisted states for the reduced system and then use the broadcasting approach to extend this \revision{to the network of networks}.

We consider here the same network as in Fig. \ref{fig:3_layers_random_ic}: $3$ \revision{sub-networks}, where each \revision{one} is composed of $N = 100$ nodes, and the internal connection structure of each \revision{area} is given \revision{by a random network. Here, the internal coupling is given by $\epsilon_{\mathrm{intra}} = 1.0$ and the coupling between sub-networks is given by $\epsilon_{\mathrm{inter}} = 0.01$}. The reduced system is then given by a $3$ nodes network, which is coupled in an all-to-all scheme. A possible phase-locking state for this network, following Eq. (\ref{eq:twisted_states}) is represented in Fig. \ref{fig:3_layers_twisted_state}a, where constant difference of $\sfrac{2\pi}{3}$ is observed. As stated before, all oscillators within a layer have the same phase, so each \revision{area} is phase synchronized, but the entire system, in this particular case, has $R = 0$, given the phase-locking solution.
\begin{figure*}[htb]
    \centering
    \includegraphics[width=0.9\textwidth]{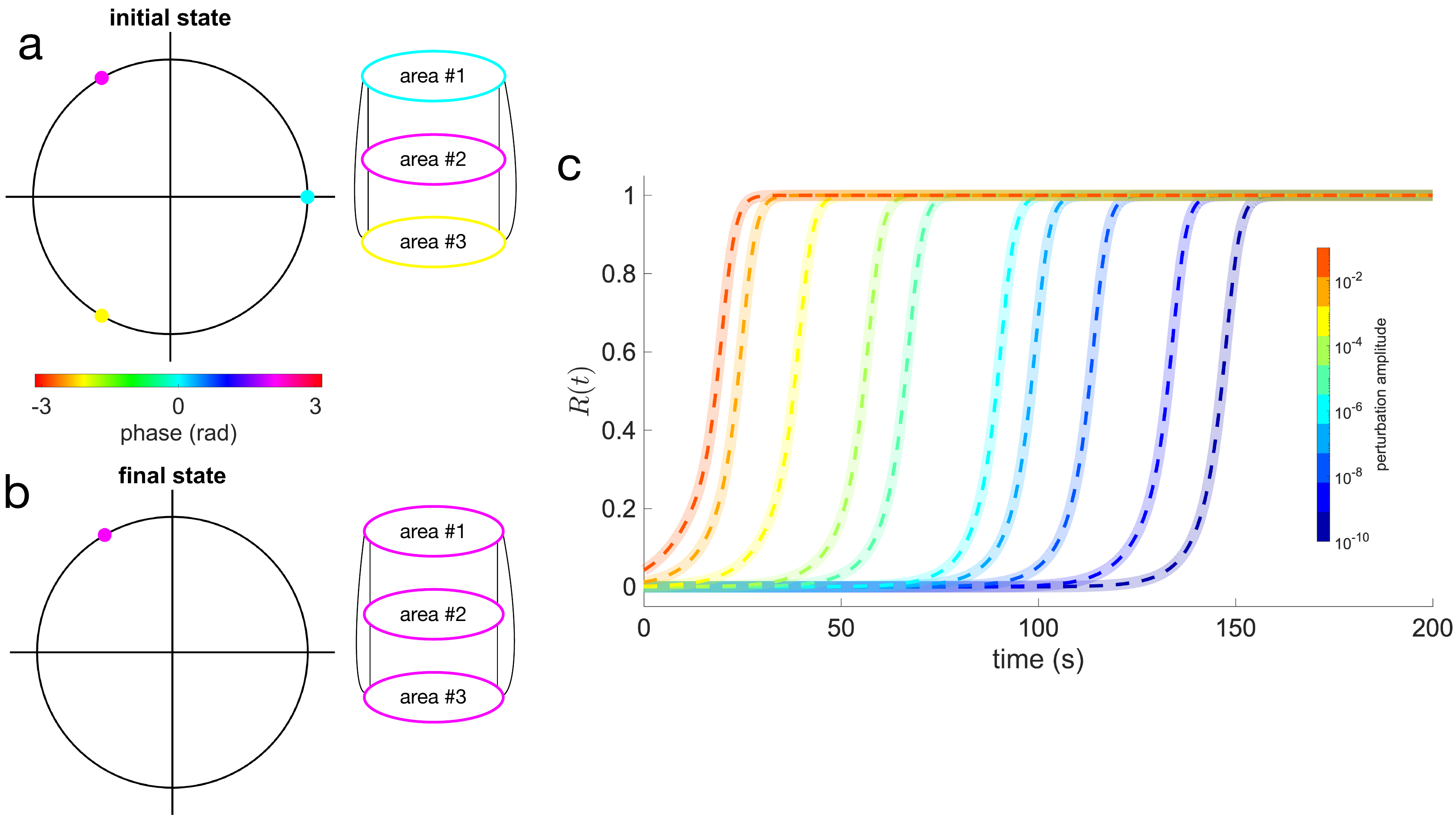}
    \caption{\textbf{Dynamics on a \revision{network of networks} and its reduced system on a phase-locking state.} \textbf{(a)} We consider a twisted state for the reduced system following Eq. (\ref{eq:twisted_states}) with $M = 3$ and $p = 1$. This leads to an initial state where a constant phase difference is observed across nodes. \textbf{(b)} However, this twisted state is not stable, so as time evolves, a transition to phase synchrony is observed, and all oscillators have the same phase in the final state. \textbf{(c)} The time \revision{at which} this transition occurs, however, depends on the perturbation applied to the system. We considered different perturbations amplitudes (color-code) and analyze the Kuramoto order parameter as a function of time for the reduced system (solid-shaded lines) and for the \revision{multi-level} network (dashed lines). The bigger the perturbation the faster the transition, and our approach shows a perfect match between the dynamics of the reduced system and the \revision{multi-level} network.}
    \label{fig:3_layers_twisted_state}
\end{figure*}

This state, however, is not stable, so small perturbation, e.g. due to the coupling or perturbation on the initial state, can lead the system to transition to phase synchronization (Fig. \ref{fig:3_layers_twisted_state}b). Our approach is able to capture the details of this transition. To show this point, we perform a detailed simulation protocol, where different perturbations are applied to the initial state. This perturbation is given by a uniform, random distribution of phases, multiplied by an amplitude factor. Mathematically, the perturbation can be described by: $\mathcal{A} \, \mathcal{U}(-\pi, \pi)$, where $\mathcal{A}$ is the perturbation amplitude. Figure \ref{fig:3_layers_twisted_state}c depicts the Kuramoto order parameter as a function of time when different amplitudes are considered, which are represented in color-code. Here, the order parameter for the reduced system is shown in the solid-shaded lines, and the order parameter for the \revision{network of networks} is shown in the dashed lines. We observe that the smaller the perturbation, the longer it takes for the phase-locking state to transition to phase synchrony. We emphasize that our reduced approach has equivalent dynamics to the \revision{multi-level} system, so we observe a perfect match.

\subsection{Higher number of \revision{sub-networks}}

The approach introduced here through the reduced system and the broadcasting process is general to study \revision{network of networks} with an arbitrary number of \revision{sub-networks} and internal connection scheme. To show this point, we then analyze a \revision{multi-level} system composed of $M = 50$ \revision{sub-networks}, where each \revision{one} is given by a random network with $N = 100$ nodes. Here, all \revision{sub-networks} are connected, such that the reduced system is given by a complete graph with $M = 50$ nodes\revision{, and the internal coupling strength is given by $\epsilon_{\mathrm{intra}} = 1.0$ and inter-areas coupling varies in different simulations $\epsilon_{\mathrm{inter}} \in [0.001, 0.005]$.}

At first, we consider the case of random initial conditions for \revision{the reduced system and use the broadcast mechanism to obtain the initial conditions for the multi-level network}. Figure \ref{fig:50_layers}a show the trajectories for the reduced (solid-shaded lines) and the \revision{multi-level} systems (dashed lines) based on the phases of the oscillators ($\bar{\theta}$ and $\theta$). Due to the random initial conditions (see the inner panel for the initial state), the systems are desynchronized in the beginning of the simulation. However, as time evolves, phase synchronization is reached and all oscillators have the same phase (see inner panel for final state). This result also shows the equivalence in the dynamical behavior of both systems.
\begin{figure*}[htb]
    \centering
    \includegraphics[width=0.95\textwidth]{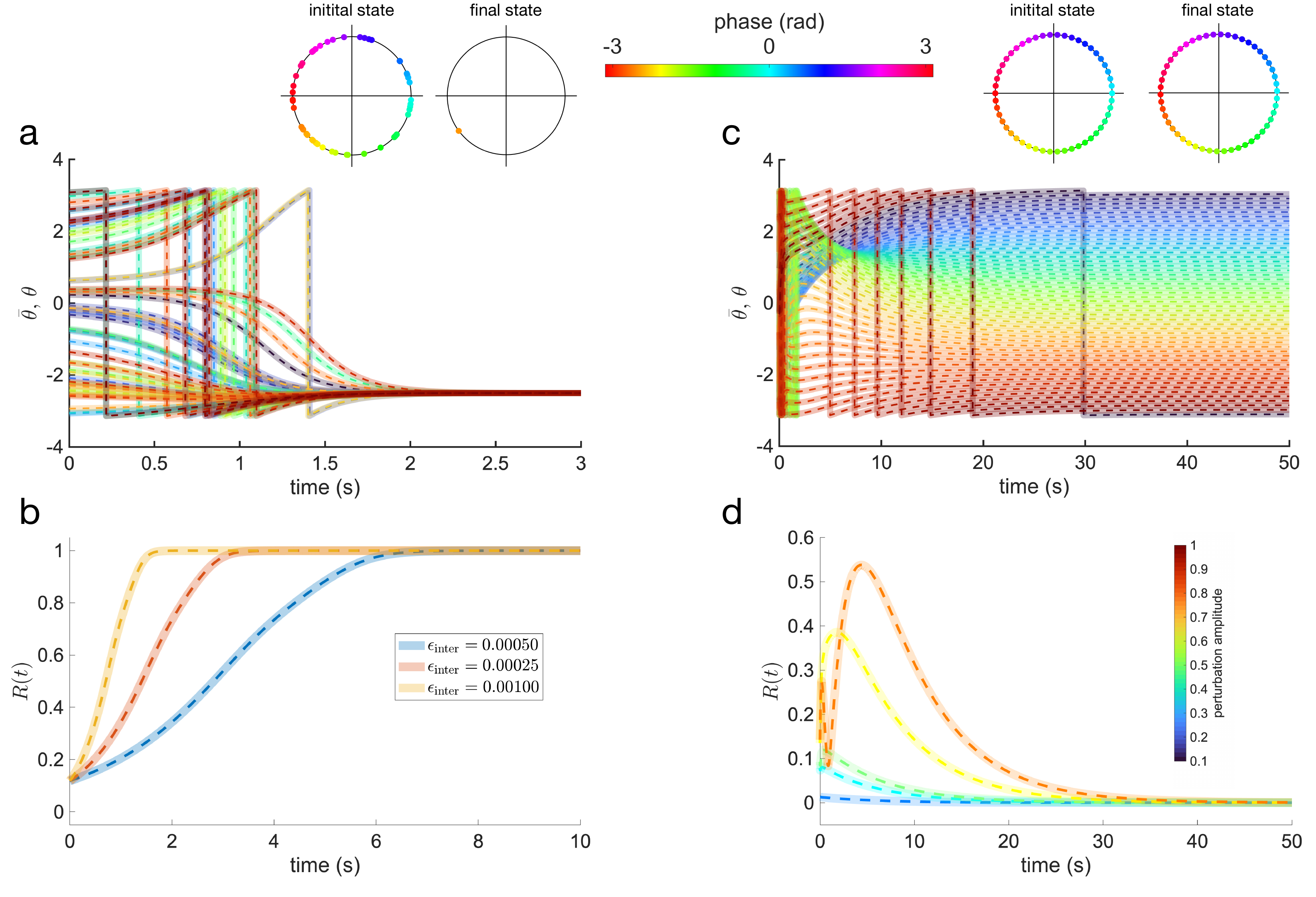}
    \caption{\textbf{Dynamics \revision{on a network with $50$ sub-networks} and its reduced system}. We first consider the case where the inter-\revision{area} coupling scheme is given by a complete graph (all-to-all). So, random initial conditions evolve to phase synchrony. Here, solid-shaded lines represent the reduced system, and dashed lines represent the \revision{network of networks}. The trajectories of the systems emphasize this point \textbf{(a)} and also show the equivalent dynamics between the systems. When we consider higher coupling strengths, the systems transition to phase synchrony quicker \textbf{(b)}. We then consider the case where the inter-\revision{area} connection scheme is given by a first-neighborhood architecture, where each \revision{area} is connected to two \revision{areas} (one in each direction) with periodic boundary conditions. In this case, phase-locking solutions are stable. We then show that the systems assume the form of a twisted state \textbf{(c)}; and that this is robust to different perturbations \textbf{(d)}.}
    \label{fig:50_layers}
\end{figure*}

Moreover, if we consider different coupling strengths between the \revision{sub-networks} (inter coupling), the transition to phase synchrony occurs at different times, as described in Prop. (\ref{prop:coupling_scale}). To show that our approach is able to capture these details, we consider the same random initial conditions and we change the inter coupling strength. Figure \ref{fig:50_layers}b shows the order parameter for the reduced system (solid-shaded lines) and for the equivalent \revision{network of networks} (dashed lines) as a function of time for different values of $\epsilon_{\mathrm{inter}}$. We observe that the higher the coupling the faster the transition to phase synchronization. \revision{Further, this follows from Prop. \ref{prop:coupling_scale}, and, in addition to that, we are able to quantify to what extent the increase on the coupling $\epsilon_{\mathrm{inter}}$ reduces the time to reach phase synchronization, given by $R(t) = 1$.}

We also use the approach introduced in this paper to analyze the broadcasting of stable phase-locking states. To do so, we consider the \revision{multi-level} system with $M = 50$ \revision{sub-networks}, each one being described by a random network with $N = 100$ nodes. Each \revision{sub-network}, however, is connected to two \revision{sub-networks} only, in a first-neighborhood fashion. This leads to a reduced system being described as a ring graph with $N = 50$ nodes, where each node has degree $2$ with periodic boundary conditions. This kind of network is known to support stable twisted states \cite{townsend2020dense}. Particularly, we use Eq. (\ref{eq:twisted_states}) with $M = 50$ and $p = 1$ to generate the stable phase-locking solution \cite{townsend2020dense} for the reduced system, which is broadcasted to the \revision{network of networks} following the approach described in this paper. Moreover, to show the stability of this solution and that our approach is able to capture these details, we apply perturbation to this phase-locking state using the same approach as before introduced: $\mathcal{A} \, \mathcal{U}(-\pi, \pi)$.

The trajectories for both, the reduced system and the \revision{network of networks} are shown in Fig. \ref{fig:50_layers}c. In this case, the initial state is given by a twisted state as shown in the inner panel. We then apply a perturbation to this state, which leads to a change in the phase configuration. However, given the stability of the $1^{\mathrm{st}}$ ($p=1$) twisted state for this system, despite the perturbation, the phase-locking is recovered and we can observe the constant phase difference across the oscillators, which leads to the final state being exactly the same as the initial one without perturbation (Fig. \ref{fig:50_layers}c inner panels). Again, the phases for the reduced system are represented by the solid-shaded lines, and the phases for the \revision{multi-level} networks are represented by the dashed lines. 

We also consider different perturbation amplitudes applied to the initial state. Figure \ref{fig:50_layers}d shows the Kuramoto order parameter for both systems (solid-shaded lines represent the reduced system and dashed lines the multilayer network) as a function of time. The perturbation amplitude is shown in color-code. We observe that the stronger the amplitude the bigger the variation of the order parameter, but given the stability of this state, in the end, $R = 0$ for all cases, which characterizes the solution for these systems.

\section{Discussions and conclusions}\label{sec:conclusions}

In this paper, we have introduced an alternative approach to study and predict the dynamical behavior of Kuramoto oscillators on a class of \revision{networks of networks}. Based on the results presented in \cite{doan2022joins, doan2022sprectrum}, we can represent \revision{these networks} as a join of matrices, each one representing a different connection between nodes either in the same \revision{sub-network} or in different \revision{sub-networks}. From this representation, we have developed a ``reduced system" that holds the important information regarding the original, \revision{multi-level} network. While the \revision{network of networks} is composed of $M$ \revision{sub-networks} with $N$ nodes, the reduced system is given by $M$ elements. We have shown that, when the initial state of the \revision{multi-level} system has a relation with the initial state of the reduced one, both systems have equivalent dynamics. Thus, it allows us to investigate the dynamical behavior of the \revision{network of networks} in a simpler way.

\revision{This problem has sparked a lot of interest due to its wide range of applications. The idea of finding a reduced, simpler version of a sophisticated system has been explored focusing on different aspects of multi-level networks and dynamics \cite{kivela2014multilayer,boccaletti2014structure,schaub2016graph,menara2019stability,tiberi2017synchronization, gao2011robustness,gao2012networks,kenett2015networks}. In our paper, we have explored this problem in order to contribute to this discussion. Our results have shown that} we can find solutions for the reduced system and broadcast them to the \revision{network of networks}. This method offers an alternative and simple way to find equilibrium points for Kuramoto oscillators on \revision{these networks}. This approach is general to arbitrary topologies for intra-\revision{areas} connections (within each \revision{sub-network}). \revision{Importantly, we have extended the discussion on multi-level networks and dynamics regarding the stability of a given solution. In our paper, we have shown that} we can use \revision{the approach we have introduced} to obtain information on the linear stability of equilibrium points \revision{in multi-level networks}. We can write the Jacobian matrix for both systems\revision{, the multi-level and the reduced one,} and, by using the results from \cite{doan2022joins}, we can obtain the spectrum of the Jacobian of the \revision{multi-level} network based on the spectrum of the Jacobian of the reduced system. Therefore, we are now able to investigate the linear stability of equilibrium points of \revision{network of networks} in a simple way.

\revision{At the same time that our approach contributes to the discussion of the dynamics in multi-level systems, however, there are several open problems yet to be solved. The approach we have shown in this paper was studied in a simple version of a multi-level system. In this way, the extension of this approach to multilayer networks with different coupling functions and different individual dynamics, and the consideration of different dynamical states than phase synchronization within each sub-network are important features that would contribute to the study of the dynamical behavior of multi-level systems.}

\revision{This is an important problem in modern science, since these} networks have been used to model a variety of systems. Extensive numerical studies have shown a rich repertory of dynamical behavior \cite{kasatkin2017self,kachhvah2021explosive, majhi2017chimera}. Moreover, experimental analyses confirm this feature \cite{blaha2019cluster, leyva2017inter}. Furthermore, the use of \revision{multi-level} networks has been helpful in the understanding of complex systems, e.g. neural systems \cite{bassett2018nature,palmigiano2017flexible}. However, analytical treatment for this kind of system is still an open problem in these diverse fields. Our approach thus offers a novel path in the study of \revision{multi-level} network, opening the possibility of analytical and mechanistic insights into the dynamics of these important systems.

\revision{\section*{Code availability}

An open-source code repository for this work is available on GitHub: \href{http://mullerlab.github.io}{\textcolor{Cerulean}{mullerlab.github.io}}.}

\begin{acknowledgments}
This work was supported by BrainsCAN at Western University through the Canada First Research Excellence Fund (CFREF), the NSF through a NeuroNex award (\#2015276), the Natural Sciences and Engineering Research Council of Canada (NSERC) grant R0370A01, ONR N00014-16-1-2829, SPIRITS 2020 of Kyoto University, Compute Ontario (computeontario.ca), Compute Canada (computecanada.ca), and by the Western Academy for Advanced Research. J.M.~gratefully acknowledges the Western University Faculty of Science Distinguished Professorship in 2020-2021. R.C.B gratefully acknowledges the Western Institute for Neuroscience Clinical Research Postdoctoral Fellowship. R.D. was supported by the Swiss National Science Foundation, under grant number P400P2\_194359. \revision{We thank the anonymous Referees for their help in improving the quality and clarity of this paper.}
\end{acknowledgments}

%

\end{document}